\documentclass[reqno,A4paper]{amsart}

\usepackage{amsmath}
\usepackage{amssymb}
\usepackage{amsthm}
\usepackage{enumerate}
\usepackage{mathrsfs} 
                             \usepackage{graphicx}
\usepackage{eqlist}
\usepackage{array}

\newtheorem{theorem}{Theorem}[section]
\newtheorem{prop}[theorem]{Proposition}
\newtheorem{lemma}[theorem]{Lemma}
\newtheorem{corol}[theorem]{Corollary}
\newtheorem{notation}[theorem]{Notation}

\newtheorem{problem}[theorem]{Problem}

\theoremstyle{definition}
\newtheorem{definition}[theorem]{Definition}
\newtheorem{remark}[theorem]{\textup{Remark}} 
\newtheorem{example}[theorem]{\textit{Example}} 

\numberwithin{equation}{section}

\begin{document}

\title[On Quasi-Small Loop Groups]%
{On Quasi-Small Loop Groups}
\author[B.~Mashayekhy, H.~Mirebrahimi, H.~Torabi \and A.~Babaee]%
{Behrooz~Mashayekhy{$^1$}, Hanieh~Mirebrahimi{$^2$}, Hamid~Torabi{$^3$} \and Ameneh~Babaee{$^4$}}

\newcommand{\acr}{\newline\indent}

\address{\llap{1\,}Department of Pure Mathematics\acr Center of Excellence in Analysis on Algebraic Structures\acr Ferdowsi University of
Mashhad\acr 
P.O.Box 1159-91775\acr Mashhad, Iran}
\email{bmashf@um.ac.ir}

\address{\llap{2\,}Department of Pure Mathematics\acr Center of Excellence in Analysis on Algebraic Structures\acr Ferdowsi University of
Mashhad\acr 
P.O.Box 1159-91775\acr Mashhad, Iran}
\email{h\_mirebrahimi@um.ac.ir}

\address{\llap{3\,}Department of Pure Mathematics\acr Center of Excellence in Analysis on Algebraic Structures\acr Ferdowsi University of
Mashhad\acr 
P.O.Box 1159-91775\acr Mashhad, Iran}
\email{h.torabi@um.ac.ir}

\address{\llap{4\,}Department of Pure Mathematics\acr Center of Excellence in Analysis on Algebraic Structures\acr Ferdowsi University of
Mashhad\acr 
P.O.Box 1159-91775\acr Mashhad, Iran}
\email{am.babaee@mail.um.ac.ir}




\begin{abstract}
In this paper, we study some properties of homotopical closeness for paths. We define the quasi-small loop group as the subgroup of all classes of loops that are homotopically close to null-homotopic loops, denoted by $\pi_1^{qs} (X, x)$ for a pointed space $(X, x)$. Then we prove that, unlike the small loop group, the quasi-small loop group $\pi_1^{qs}(X, x)$ does not depend on the base point, and that it is a normal subgroup containing $\pi_1^{sg}(X, x)$, the small generated subgroup of the fundamental group. Also, we show that a space $X$ is homotopically path Hausdorff if and only if $\pi_1^{qs} (X, x)$ is trivial. Finally, as consequences, we give some relationships between the quasi-small loop group and the quasi-topological fundamental group.
\end{abstract}

\maketitle

\section{Introduction and Motivation}

Paths in   topological spaces interpret how to transfer from a point to another one. The path space is the set of all paths in a given topological space equipped with some topology, which consists of some information about transferring in  topological spaces. For instance, in configuration spaces ---some specific topological spaces which are used in physics, mechanics, robotics, and so on---, a path means a stable motion from a configuration to another one. Thus, the study of paths and path spaces of a topological space, is useful to make some computations easier in other fields  (see \cite{Far}). In this paper, we study a relation between paths, called homotopical closeness, which presents some properties of the  topology of the path space. There are several topologies defined on the path space, each of which has some information and applications. 
The notion of closeness is a generalization of the   smallness studied in \cite{3} for closed paths called loops. Some specific loops have information about some properties of the topological space, path space, and local properties of the space; see \cite{KarRep}.
Virk \cite{3} defined the small loop as a loop $\alpha$ such that for every open neighborhood $U$ of $\alpha(1)$, there exists a loop  homotopic to $\alpha$ contained in $U$.   
Also, Virk \cite{2} assumed the smallness as a special case of closeness; that is, a loop is small whenever it is close to the constant loop. Moreover, Virk \cite[Definition 58]{2} introduced the concept of homotopical closeness for two maps $f,g:K\to X$ on a compact Hausdorff space $K$ such that $f\not \simeq  g$. Here we recall  Virk's definition of closeness in the case of paths with some modifications because of its technical results. 
\begin{definition}\label{ASL}
Let $f $ and $g$ be two paths in $X$ with $f (0)=g (0)$ and $f (1)=g(1)$. We say $f$ is \textit{homotopically close} to $g$ relative to $\dot{I}$ (denoted by $f \xrightarrow{close}g$ rel  $\dot{I}$), if for every partition $0=t_0 < t_1 < \cdots < t_n =1$ and every sequence of open sets $U_1 ,\ldots ,U_n$ with $g ([t_{i-1},t_i ])\subset U_i$, there exists a path $\gamma$ satisfying $\gamma ([t_{i-1},t_i ])\subset U_i$ for $1\leq i\leq n$ and $\gamma (t_i )=g (t_i )$ for $0\leq i\leq n$ such that $\gamma \simeq f$ rel $\dot{I}$. 
\end{definition}
In this paper, closeness will always refer to Definition \ref{ASL}. Although, Definition \ref{ASL} may seem to be a special case of Definition 58 of Virk \cite{2}, there are two major differences.
First, if $f\simeq g$ rel $\dot{I}$, then by Definition 58 of Virk \cite{2} $f$ is not close to $g$, but by Definition \ref{ASL}, $f$ is homotopically close to $g$ relative to $\dot{I}$.
%
In fact, we require trivial cases to form a group structure. Moreover, to obtain a normal subgroup structure, the condition of being equal at the points $t_i$ is added;  that is, $\gamma (t_i )=g (t_i )$ for $0\leq i\leq n$.

In Section 2, we investigate basic properties of homotopical closeness of paths, for instance preserving by continuous maps, preserving by the concatenation of loops, and so forth. We need these properties to define a subgroup of the fundamental group, called quasi-small loop group as a generalization of the small loop group. The fundamental group $\pi_1(X, x)$ for a pointed space $(X, x)$, the set of all homotopy classes of loops, can be considered as a subset of the quotient of the path space induced by the equivalence relation of homotopy between paths. The fundamental group is a useful tool to classify and study topological spaces. Moreover, it can help to obtain some properties of the path space. Also, some specific subgroups, such as small loop group, small generated subgroup, Spanier subgroup, and so on of the fundamental group, are used as new  tools to study topological spaces. These subgroups were used to obtain information about covering spaces, path spaces, topologized fundamental group, the local properties of the topological spaces, and so on; see \cite{BabMas3, BabMas4, 31}. For instance, Virk \cite{3} introduced the small loop
group $\pi_1^s (X, x)$ as a subgroup of the fundamental group $\pi_1 (X, x)$, consisting of all homotopy classes of small loops, and studied its impact on covering spaces. The small loop group is a subgroup of $\pi_1(X, x)$, but it is not necessarily a normal subgroup.

By the definition of smallness,  a non-trivial loop $\alpha$ at $x$ is small if and only if $\alpha \xrightarrow{close} c_x$. 
Therefore we can restate $\pi_1^s (X,x)$ as follows:
$$
 \pi_1^s (X,x)=\{ [f]\in \pi_1 (X,x)\; |\; f\xrightarrow{close} c_x \}.
$$
In general, $\pi_1^s(X, x)$ may admit different structures at different points. In order to have a subgroup not depending on the base point, Virk \cite{3} introduced the small generated subgroup, denoted by $\pi^{sg}_1 (X, x)$, as the subgroup generated by the set
$$
\{ [\beta *\alpha *\beta^{-1}]\; |\; \beta \in P(X,x), \; [\alpha ]\in \pi_1^s (X,\beta (1)) \} ,
$$
where $P(X, x)$ is the space of all paths in $X$ with initial point $x$. In other words,
$$
\pi_1^{sg}(X, x) = \langle \{ [\beta *\alpha *\beta^{-1}]\; |\; \beta \in P(X,x), \; \alpha \xrightarrow{close} c_{\beta (1)} \} \rangle.
$$
In Lemma \ref{02},  we show that $\alpha \xrightarrow{close} c_{\beta (1)}$ if and only if $\beta *\alpha *\beta^{-1} \xrightarrow{close} \beta *c_{\beta (1)}*\beta^{-1}$.
Obviously for $\beta \neq c_x$, we have $\beta *c_{\beta (1)}*\beta^{-1}\neq c_{x}$, but $\beta *c_{\beta (1)}*\beta^{-1} \in [c_x ]$. This property motivates us to generalize the small generated subgroup to the following subgroup studied in Section 3.
\begin{definition}\label{def1}
Let $X$ be a topological space and let $x \in X$. 
We define the set $\pi_{1}^{qs}(X,x)$ as follows:
\[
\pi_1^{qs} (X,x)=\{ [f]\in \pi_1 (X,x)\; |\; f\xrightarrow{close}f' , \text{for some}\;  f'\in [c_x]\} .
\]
We see that the set $\pi_{1}^{qs}(X,x )$ is a normal subgroup of $\pi_1(X, x)$ (see Theorem \ref{th3.2}), which we call the \textit{quasi-small loop group} of $(X,x)$. \end{definition}
In Section 3, for an arbitrary subset $H$ of the fundamental group $\pi_1(X, x)$, the quasi-small loop group is generalized to the $H$-quasi-small loop subset, denoted by $\pi_H^{qs} (X, x)$, as the subset consisting of all classes of homotopically  close loops to some loop whose class belongs to $H$. 
We prove the following statements for the $H$-quasi-small loop group:
\begin{itemize}
\item
 $\pi_H^{qs}(X, x)$ is a (normal) subgroup of the fundamental group $\pi_1(X, x)$, if $H$ is a (normal) subgroup of $\pi_1(X, x)$ (Theorem \ref{th3.2}).
\item
For any subgroup $H$, $\pi_H^{qs} (X, x) = H \pi_1^{qs}(X, x)$ (Proposition \ref{4444}). Therefore, if $H$ contains $\pi_1^{qs} (X, x)$, then $\pi_H^{qs} (X, x) = H$ (Corollary \ref{co3.4}).
\end{itemize}

In this paper, we present some spaces with $\pi_1^s (X,x)=1$ while they  have  loops that behave somehow as if they are small loops, that is, $\pi_1^{qs} (X,x)\neq 1$ (see Figures \ref{fig2} and  \ref{fig1}). Moreover,  we prove that, unlike $\pi_1^s  (X,x)$, the subgroup $\pi_1^{qs} (X,x)$ is independent of the base point. 
We prove the following sequence of inequalities (Proposition \ref{relation}):
\begin{equation}\tag{\ref{eqmain}}
\pi_{1}^{sg}(X,x)\leqslant \pi_1^{qs} (X,x)\leqslant \pi_{1}^{sp}(X,x).
\end{equation}

Moreover, by some examples, we show that the above inclusions may be strict (Remark \ref{rem1}). One of the example was constructed  in \cite[p. 370]{2} by modifying the Harmonic archipelago space constructed by attaching cells to some one-dimensional space. The Harmonic archipelago is a well-known space obtained by attaching cells to Hawaiian earring, whose homology and cohomology groups were calculated by Karimov-Repov\v{s} \cite{KarRephom}. There are different ways for attaching cells, and each way may yield a unique space (see Eda-Karimov-Repov\v{s} \cite{EdaKar} and Eda et al. \cite{EdaSnake}).
Some of which were constructed to be  counterexamples for some natural conjectures (see Karimov et al. \cite{KarRep} and Male\v{s}i\v{c} et al. \cite{MalRep}).
Some  spaces obtained by cells are called cell-like spaces and they have unexpected behavior (Eda-Karimov-Repov\v{s} \cite{EdaKar2n, EdaKar3n}). For more information about these spaces see Eda-Karimov-Repov\v{s} \cite{EdaKar4n, EdaKar5n}.

Hausdorffness, the second condition of separability,  is the ability to separate points by disjoint open sets, and  it is an  essential condition to verify metrizability of the spaces.  Homotopical Hausdorffness, introduced in \cite{Con}, as named, is  the ability to separate paths and loops. The original homotopical Hausdorffness is not to have a small loop and it is equivalent to the path space being Hausdorff; see \cite{FisZas}. Homotopical Hausdorffness was generalized and modified in several ways; some of which are equivalent to some specific path spaces being Hausdorff \cite{BabMas3, BroDyd, FisZas}.  
The concepts of homotopical smallness and closeness are related to various versions of the property of  homotopical Hausdorffness (Brazas-Fabel \cite{1}, Cannon-Conner \cite{Can}, and Conner et al.  \cite{Con}).  
In this paper, we discuss homotopical path Hausdorffness defined as follows.

\begin{definition}\label{def10}
A topological space $X$ is called
\begin{enumerate}
\item
\textit{homotopically Hausdorff} if for every $x \in X$ and for every non-trivial $\gamma \in \pi_1 (X, x)$, there exists a neighborhood
$U$ of $x$ such that no loop in $U$ is homotopic to $\gamma$ in $X$;
\item
\textit{homotopically path Hausdorff relative to a subset}  $H\subseteq \pi_1 (X, x)$ if for every pair of paths $\alpha, \beta \in P(X, x)$ such that $\alpha (1)=\beta (1)$ and $[\alpha * \beta^{-1}]\not \in H$, there are a partition $0 = t_0 < t_1 < t_2 < \cdots < t_n = 1$ and a sequence of open sets $U_1,U_2, \ldots ,U_n$ with $\alpha ([t_{i-1}, t_i ]) \subset U_i$, such that if $\gamma : [0, 1] \to 
X$ is another path satisfying $\gamma ([t_{i-1}, t_i ]) \subset U_i$ for $1 \leq  i \leq  n$ and $\gamma (t_i ) = \alpha (t_i )$ for
$0 \leq i \leq n$, then $[\gamma * \beta^{-1}] \not \in H$.
\item
\textit{homotopically path Hausdorff }if it is homotopically path Hausdorff relative to the trivial subgroup $H = 1$.
\end{enumerate}
\end{definition}

The property of homotopical path Hausdorffness is important for distinguishing the classes of loops and also for the existence of a certain generalized covering \cite{1}. We find a close relation between being homotopically path Hausdorff and quasi-small loop group of a space $X$. We prove that a space $X$ is homotopically path Hausdorff  if and only if $\pi_1^{qs} (X, x) = 1$ (Theorem \ref{s}). Then, for a homotopically path Hausdorff space $X$, we obtain that $\pi_H^{qs} (X, x) = H$ for any subgroup $H$ (Theorem \ref{thm100} (2)).
It is a generalization of the well-known fact: $X$ is homotopically Hausdorff if and only if $\pi_1^s(X,x) =1$. Previously, some other conditions were found for a space to be  homotopically (path) Hausdorff relative to $H$ by using the topologized fundamental group \cite{BabMas3, PasMas}. Also, by the structures of Hawaiian groups, the authors  \cite{BabMas4} presented some equivalent conditions for a space to be homotopically Hausdorff relative to $H$.

The fundamental group contains some topological properties of the spaces, and  topologists intend to equip the fundamental group with some topologies to use it as a strong tool. Various topologies exist on the fundamental group and the path space; each of which has its properties and applications in covering theory, Hawaiian groups, and so on; see \cite{BabMas3, 1, KarRep}. Moreover, topologies of the fundamental group reveal some properties of the spaces such as local properties, the classification of the spaces, and so on. The topologized fundamental group was investigated by closed subgroups, open subgroups, convergent sequences, and so on; see \cite{BabMas3, 1, PasMas}. Some subgroups, such as small subgroup, small generated subgroup, and Spanier subgroup, have strategic roles in studying some topologies of the fundamental group. In this paper, we consider the quasi-topological group recalled by Brazas-Fabel \cite{1}, and  we give some relationships between the quasi-small loop group and the quasi-topological fundamental group $\pi_1^{qtop} (X,x)$. In fact,  it was proved in \cite{1} that the following properties are  equivalent:
\begin{itemize}
\item
$X$ is homotopically path Hausdorff; 
\item
The trivial subgroup is closed in $\pi_1^{qtop} (X, x)$;
\item
$\pi_1^{qtop} (X, x)$ satisfies the first separability axiom, $T_1$ axiom.
\end{itemize}
In this paper, we present an equivalent condition for the above equivalences, by using algebraic structures; that is, $\pi_1^{qtop} (X, x)$ satisfies the $T_1$ axiom if and only if
$\pi_1^{qs}(X, x) = 1$.


Throughout this paper, every topological space is assumed to be path connected. Also,  by $[\cdot]$ we mean the homotopy class of paths relative to the boundary of unit interval $\dot{I}$.

\section{Homotopical closeness}
The aim of this section is to investigate some basic properties of homotopical closeness of paths in topological spaces. These properties help us to study the quasi-small loop group in the next section. First note that for locally path-connected spaces, Definition \ref{ASL} and Definition 58 of Virk \cite{2} for paths (except for homotopic ones) coincide as follows.

\begin{remark}\label{re2.1}
Let $X$ be a locally path-connected space. If two paths $f$ and $g$  in $X$ are not homotopic, then $f$ is close to $g$ in the sense of Definition 58 of Virk \cite{2}, if and only if $f$ is homotopically close to $g$ in the sense of Definition \ref{ASL}.
In order to prove this, assume that $f$ is close to $g$ in the sense of Definition 58 of Virk \cite{2}. Then for the sequence of open sets $U_1, \ldots, U_n$, there exist a partition $0= t_0 < t_1<\cdots < t_n = 1$ with $g([t_{i-1}, t_i]) \subseteq U_i$ and a map $f' \simeq f$ rel $\dot{I}$ with $f'([t_{i -1}, t_i]) \subseteq U_i$ for $1 \le i \le n$. Since $X$ is locally path connected, there exists a path $\lambda_i$ from $f'(t_i)$ to $g(t_i)$ in $U_i$ for $1 \le i \le n-1$. Define $\gamma|_{[t_0, t_1]} = f'|_{[t_0, t_1]} * \lambda_1$, $\gamma|_{[t_{i-1}, t_i]} =\lambda_{i-1}^{-1} * f'|_{[t_{i-1}, t_i]} * \lambda_i$  for $2\le i \le n-1$, and $\gamma|_{[t_{n-1}, t_n]} = \lambda_{n-1}^{-1} * f'|_{[t_{n-1}, t_n]}$. One can check that $\gamma$ satisfies Definition \ref{ASL}.
For non-locally path-connected spaces, 
Definition \ref{ASL} differs from Definition 58 of Virk \cite{2} as stated in Example \ref{ex2.2}.
\end{remark}

The closeness of maps is not topologically invariant, as shown by the  example in \cite[p. 368]{2}.  In the following proposition, using the condition $\gamma(t_i)  = g(t_i)$ in Definition \ref{ASL}, we prove that  continuous maps preserve homotopical closeness for paths (see \cite[Proposition 46]{2} and \cite[Corollary 47]{2}).
\begin{prop}\label{pr2.3}
Assume that $f$ and $g$ are paths in $X$ with $f(0)=g(0)$ and  $f(1)=g(1)$ and that $\phi : X \to Y$ is a  continuous map. If $f\xrightarrow{close}g$ rel $\dot{I}$, then $\phi f \xrightarrow{close}\phi g$ rel $\dot{I}$.
\end{prop}
\begin{proof}
Consider a partition $0=t_0 < t_1 < \cdots < t_n =1$ and a sequence of open sets $U_1 ,\ldots ,U_n$ of $Y$ with $\phi g ([t_{i-1},t_i ])\subset U_i$. Then $g([t_{i-1},t_i ])\subset \phi^{-1}(U_i)$ for $1\leq i\leq n$. Since $f\xrightarrow{close}g$ rel $\dot{I}$, there exists a path $\gamma :I\to X$ satisfying $\gamma ([t_{i-1},t_i ])\subset \phi^{-1}(U_i)$ for $1\leq i\leq n$ and $\gamma (t_i )=g (t_i )$ for $0\leq i\leq n$ such that $\gamma \simeq f$ rel $\dot{I}$. Hence we have $\phi\gamma ([t_{i-1},t_i ])\subset (U_i)$ for $1\leq i\leq n$ and $\phi\gamma (t_i )=\phi g (t_i )$ for $0\leq i\leq n$ with $\phi\gamma \simeq \phi f$ rel $\dot{I}$. This means $\phi f\xrightarrow{close}\phi g$ rel $\dot{I}$.
\end{proof}

By the definition of homotopical closeness of paths, one can easily prove the following statement.

\begin{lemma}\label{01}
Assume that   $f_0 ,f_1 ,g$ are  paths in $X$ with  $f_0 (0)=f_1 (0)=g (0)$ and $f_0 (1)=f_1 (1)= g(1)$. If $f_0 \xrightarrow{close} g$  rel $\dot{I}$ and $
f_0 \simeq f_1$ rel $\dot{I}$,   then $f_1 \xrightarrow{close} g$ rel $\dot{I}$.
\end{lemma}

\begin{remark}\label{000}
Assume that   $f  ,g_0 ,g_1 $ are  paths in $X$ with  $f (0)=g_0 (0)=g_1 (0)$ and $f (1)=g_0 (1)=g_1 (1)$. 
\begin{enumerate}
\item
The conditions $f \xrightarrow{close}g_0$ rel $\dot{I}$ and $g_0 \simeq g_1 $ rel $\dot{I}$ do not imply that $f \xrightarrow{close}g_1 $ rel $\dot{I}$;

Consider the space $C(\mathbb{S}^1 ,\{ 0\})$ introduced by Virk \cite[p. 370]{2}, which is a modification of the Harmonic archipelago space. The space $C(\mathbb{S}^1 ,\{ 0\})$ is obtained by attaching $2$-cells to a disjoint union of countably infinite number of circles with a point in common. The radii of circles tend to $1$, the circle of radius $1$ is denoted by $S_{\infty}^1$, and each $2$-cell is attached between any two consecutive circles (see Figure \ref{fig2}). The attached $2$-cells have humps (the subspaces of cells consisting of all the points having the most $z$-coordinate), which converge to entire $\big( S^1_{\infty} - \{0\} \big) \times \{1\}$, so that $C(\mathbb{S}^1,{0})$ is not locally path connected at any point of $S^1_{\infty} - \{0\}$. We attach a 2-cell $e^2$ to $S^{1}_{\infty}$   by which the loop $f'$  (indicated  in the Figure  \ref{fig2}) is null-homotopic. We denote the resulting space by $Z$ (Figure \ref{fig2}).  Clearly $f\xrightarrow{close}f' $ rel $\dot{I}$ and $f'\in [c_x]$. Put $g_0 =f'$ and $g_1 =c_x$. 

\item
The conditions $f \xrightarrow{close}g_0$ rel $\dot{I}$ and $g_0 \xrightarrow{close}g_1$ rel $\dot{I}$ do not imply that  $f \xrightarrow{close}g_1 $ rel $\dot{I}$. 
By Definition \ref{ASL}, $g_0 \simeq g_1 $ rel $\dot{I}$ implies $g_0 \xrightarrow{close} g_1$ rel $\dot{I}$. So it is enough to take $f, g_{0}, g_{1}$ as item (1).

\item
The condition $f \xrightarrow{close}g_0$ rel $\dot{I}$ does not imply that $g_0 \xrightarrow{close}f $ rel $\dot{I}$. Equivalently, homotopical 
closeness is not symmetric (see Remark \ref{re2.1} and \cite[p. 369]{2}). 


\end{enumerate}
\end{remark}

\begin{figure}[tb]
\centering
\includegraphics[scale=0.4]{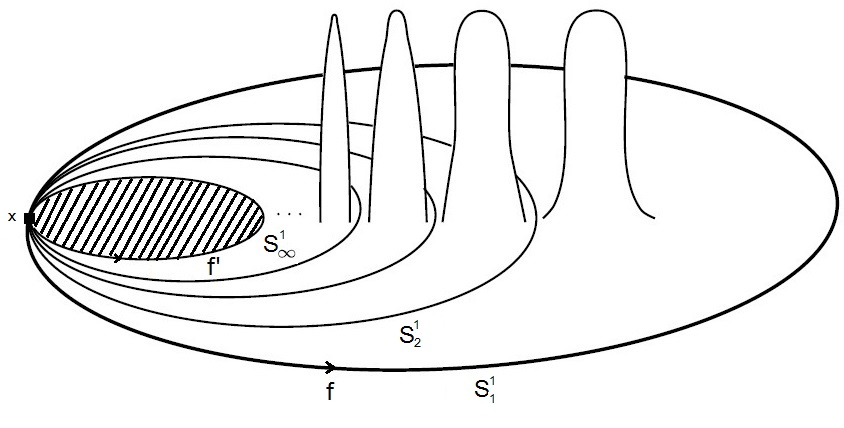}
\caption{The space $Z$}\label{fig2}
\end{figure}
We recall  the notion of the Spanier group as presented in Fischer et al. \cite{On} as follows.
\begin{definition}[\cite{On}]
Let $\mathcal{U} = \{ U_i \; |\; i\in I \}$ be  an arbitrary open covering of $X$. Then define $\pi (\mathcal{U}, x)$ to be   the subgroup of $\pi_1 (X, x)$ that contains all homotopy classes having representatives of the following type:
$$
\prod_{j=1}^{n}u_j *v_j * u^{-1}_{j},
$$
where $u_j$'s are paths (starting at the base point $x$) and each $v_j$ is a loop inside one of the neighborhoods $U_i \in \mathcal{U}$. The subgroup $\pi (\mathcal{U}, x)$ is called the Spanier group of $(X, x)$ with respect to $\mathcal{U}$.
Let $\mathcal{U}$ and $\mathcal{V}$ be open covers of $X$ and let $\mathcal{U}$ be a refinement of $\mathcal{V}$. Then $\pi (\mathcal{U}, x) \subset \pi (\mathcal{V}, x)$. This inclusion induces an inverse limit defined via the directed system of all covers with respect to the refinement. Such limit is called  the (unbased) \textit{Spanier group} of the space $X$, and we denote it by $\pi_{1}^{sp}(X,x)$.
\end{definition}
Virk  \cite{2} mentioned that close loops have no influence on  groups $\pi_1^s$ and $\pi_{1}^{sg}$  but may \textit{interfere with} $\pi_{1}^{sp}$. In the following proposition, we state the relationship between closeness and the Spanier group (compare with \cite[Proposition 50]{2}, in the case of locally path connected spaces). 
\begin{prop}\label{SP}
If $f\xrightarrow{close}g $ rel $\dot{I}$, then $[f*g^{-1}] \in \pi_{1}^{sp}(X,x)$.
\end{prop}
\begin{proof}
To prove the claim, we partially imitate the proof of  \cite[Proposition 50]{2}. Fix a cover $\mathcal{U}$ of $X$ and choose a finite subfamily
$U_1, \ldots , U_n \subset \mathcal{U}$ covering $g([0, 1])$ so that for some partition $0 = t_0 < t_1 < \ldots < t_n = 1$, the set $U_i$ contains $g([t_{i-1}, t_i ])$,  for all $i$. Since $f\xrightarrow{close}g $ rel $\dot{I}$, there exists a path $\gamma :I\to X$ satisfying $\gamma ([t_{i-1},t_i ])\subset U_i$ for $1\leq i\leq n$ and $\gamma (t_i )=g (t_i )$ for $0\leq i\leq n$ such that $\gamma \simeq f $ rel $\dot{I}$. Observe that the oriented loop $Q_j$ defined as a
concatenation $\gamma |_{[t _{i-1},t_i ]} * (g|_{[t_{i-1},t_i ]})^{-1}$ is based at $g(t_i)$ and contained in $U_i$ . The class $[ \gamma * g^{-1}]$ is contained in $\pi_1 (U, x)$ because it can be expressed as
$$
\prod_{i=1}^{n}g|_{[0,t_{i-1}]}*Q_i *(g|_{[t_{0},t_{i-1}]})^{-1}.
$$
Hence $[f*g^{-1}]=[\gamma *g^{-1}]\in \pi_{1}^{sp}(X,x)$.
\end{proof}
The converse of Proposition \ref{SP} does not hold in general. To see this, consider loops $[f],[g]\in \pi_1 (X,x)$ for which $f\xrightarrow{close}g\; rel\; \dot{I}$ but $g\not \xrightarrow{close}f\; rel\; \dot{I}$ (see item (3) of  Remark \ref{000}). Since $f\xrightarrow{close}g\; rel\; \dot{I}$,  Proposition \ref{SP} implies that $[f][g]^{-1}\in \pi_1^{sp}(X,x)$, and so $[g][f]^{-1}\in \pi_1^{sp}(X,x)$. If the converse holds, since $[g][f]^{-1}\in \pi_1^{sp}(X,x)$, then we must have $g\xrightarrow{close}f\; rel\; \dot{I}$, which is a contradiction. 
\begin{lemma}\label{02}
Assume that   $f_0 ,f_1 ,g_0 ,g_1$ are some paths in $X$ with  
$$
f_0 \xrightarrow{close}  g_0  \text{ rel } \dot{I} \quad and\quad f_1 \xrightarrow{close} g_1 \text{ rel }\dot{I}.
$$
If $f_0 (1)=f_1 (0)=g_0 (1)=g_{1}(0)$, then $f_0 *f_1 \xrightarrow{close} g_0 *g_1 $ rel $\dot{I}$. Moreover, $f_{0}^{-1} \xrightarrow{close} g_{0}^{-1}$ rel $\dot{I}$.
\end{lemma}
\begin{proof}
Take a partition $0=t_0 < t_1 < \cdots < t_n =1$ and a sequence of open sets $U_1 ,\ldots ,U_n$ with $g_0 *g_1 ([t_{i-1},t_i ])\subset U_i$. We consider the following two cases:
\begin{enumerate}
\item
There exists $i_0 \in \{ 0,1,\ldots ,n\}$ such that $t_{i_0} =\frac{1}{2}$. Then $g_0  ([2t_{i -1},2t_{i}])= g_0 *g_1 ([t_{i -1},t_{i}])\subset U_{i}$ for $1\leq i\leq i_0$, and $g_1  ([2t_{i-1}-1,2t_{i}-1])= g_0 *g_1 ([t_{i -1},t_{i}])\subset U_{i}$ for $i_0 +1 \leq i\leq n$.  Since $f_0 \xrightarrow{close}  g_0 \; rel\; \dot{I}$, for the partition $0=2t_0 < 2t_1 < \cdots < 2t_{i_0}=1$ and  the sequence of open sets $U_1 ,\ldots ,U_{i_0}$, there exists a path $\gamma_0 :[0,1]\to X$ satisfying $\gamma_0 ([2t_{i-1},2t_i ])\subset U_i$ for $1\leq i\leq i_0$ and $\gamma_0 (2t_i )=g_0 (2t_i )$ for $0\leq i\leq i_0$ such that $\gamma_0 \simeq f_0 $ rel $\dot{I}$. Similarly, since $f_1 \xrightarrow{close} g_1 $ rel $\dot{I}$, for the partition $0=2t_{i_0}-1 < 2t_{i_0 +1} -1 < \cdots < 2t_{n}-1=1$ and the  sequence of open sets $U_{i_0 +1} ,\ldots ,U_{n}$, there exists a path $\gamma_1 :[0,1]\to X$ satisfying $\gamma_1 ([2t_{i-1}-1,2t_i -1 ])\subset U_i$ for $i_0 +1\leq i\leq n$ and $\gamma_1 (2t_i -1)=g_1 (2t_i -1)$ for $i_0 \leq i\leq n$ such that $\gamma_1 \simeq f_1 $ rel $\dot{I}$.  One can check easily that  $\gamma =\gamma_0 *\gamma_1$ satisfies $\gamma ([t_{i-1},t_i ])\subset U_i$ for $1\leq i\leq n$ and $\gamma (t_i )=g_0 *g_1 (t_i )$ for $0\leq i\leq n$ such that $\gamma \simeq f_0 *f_1 $ rel $\dot{I}$. 

\item
For some $i_0 \in \{ 0,1,\ldots ,n\}$, the condition $t_{i_0 -1} <\frac{1}{2}<t_{i_0}$ holds. Put $t_* =\frac{1}{2}$. Then $g_0  ([2t_{i -1},2t_{i}])= g_0 *g_1 ([t_{i -1},t_{i}])\subset U_{i}$ for $1\leq i\leq i_0 -1$ and also $g_0  ([2t_{i_0 -1},2t_*])= g_0 *g_1 ([t_{i_0 -1},t_*])\subset g_0 *g_1 ([t_{i_0 -1},t_{i_0}])\subset  U_{i_0}$. Now since $f_0 \xrightarrow{close}  g_0 \; rel\; \dot{I}$, for the partition $0=2t_0 < 2t_1 < \cdots < 2t_{i_0 -1}<2t_* =1$ and the sequence of open sets $U_1 ,\ldots ,U_{i_0}$, there exists a path $\gamma_0 :[0,1]\to X$ satisfying $\gamma_0 ([2t_{i-1},2t_i ])\subset U_i$ for $1\leq i\leq i_0 -1$, $\gamma_0 ([2t_{i_0 -1},2t_* ])\subset U_{i_0}$,  $\gamma_0 (2t_i )=g_0 (2t_i )$ for $0\leq i\leq i_0 -1$, and $\gamma_0 (2t_* )=g_0 (2t_* )$ such that $\gamma_0 \simeq f_0 $ rel $\dot{I}$.

On the other hand, $g_1  ([2t_{i -1}-1,2t_{i}-1])= g_0 *g_1 ([t_{i -1},t_{i}])\subset U_{i}$ for $i_0 + 1\leq i\leq n$ and $g_0  ([2t_* -1, 2t_{i_0}-1])= g_0 *g_1 ([t_* ,t_{i_0}])\subset g_0 *g_1 ([t_{i_0 -1},t_{i_0}])\subset  U_{i_0}$. Since $f_1 \xrightarrow{close} g_1 $ rel $\dot{I}$, for the partition $0=2t_* -1 <2t_{i_0}-1 < 2t_{i_0 +1} -1 < \cdots < 2t_{n}-1=1$ and the sequence of open sets $U_{i_0},U_{i_0 +1} ,\ldots ,U_{n}$, there exists a path $\gamma_1 :[0,1]\to X$ satisfying $\gamma_1 ([2t_{i-1}-1,2t_i -1 ])\subset U_i$ for $i_0 +1\leq i\leq n$, $\gamma_0 ([2t_* -1, 2t_{i_0}-1])\subset U_{i_0}$,  $\gamma_0 (2t_i -1)=g_0 (2t_i -1)$ for $i_0 \leq i\leq n$, and $\gamma_0 (2t_* -1 )=g_0 (2t_* -1)$ such that $\gamma_1 \simeq f_1 $ rel $\dot{I}$.  One can easily check that  $\gamma =\gamma_0 *\gamma_1$ satisfies $\gamma ([t_{i-1},t_i ])\subset U_i$ for $1\leq i\leq n$ and $\gamma (t_i )=g_0 *g_1 (t_i )$ for $0\leq i\leq n$ such that $\gamma \simeq f_0 *f_1 $ rel $\dot{I}$. 

The statement $f_{0}^{-1} \xrightarrow{close} g_{0}^{-1}$ rel $\dot{I}$ trivially holds. 
\end{enumerate}
\end{proof}
\begin{notation}
Assume that   $f$ and $g$ are  paths in $X$ with $f(0)=g(0)$ and $f(1)=g(1)$.   We say   $f$ is close to  $[g]$ relative to $\dot{I}$, denoted by $f\xrightarrow{close} [g]$ rel $\dot{I}$, if $f\xrightarrow{close}g' $ rel $\dot{I}$ for some $g' \in [g]$. 
\end{notation}
Note that $f\xrightarrow{close} [h]$ rel $\dot{I}$ does not imply that $f\xrightarrow{close} h'$ rel $\dot{I}$ for each $h'\in [h]$ (see item (1) of  Remark \ref{000}). Moreover, if $f\xrightarrow{close} h$ rel $\dot{I}$, then $f\xrightarrow{close} [h]$ rel $\dot{I}$, but the converse does not hold. Also, $f \xrightarrow{close}[g]$ rel $\dot{I}$ does not imply that $g \xrightarrow{close}[f]$ rel $\dot{I}$ (see item (3) of Remark \ref{000}).

\section{Quasi-small loop groups}
The  main idea of this section is to study a subgroup of $\pi_1 (X,x)$ that characterizes the property of homotopical path Hausdorffness relative to a subset of fundamental group of the space $X$.

\begin{definition}\label{def3.1n}
Let $H$ be a subset of $\pi_1 (X,x )$. We define the set $\pi_{H}^{qs}(X,x)$ as follows:
$$
\pi_{H}^{qs}(X,x)=\{ [f]\in \pi_1 (X,x)\; |\; f\xrightarrow{close} [h]\; \text{for some}\; [h]\in H \} .
$$
We call the set $\pi_{H}^{qs}(X,x )$ as the  \textit{$H$-quasi-small loop group} of $(X,x)$. 
For the trivial subgroup $H=1$, we write $\pi_{H}^{qs}(X,x)$ as $\pi_1^{qs}(X,x)$ and we call it the \textit{quasi-small loop group} of $(X,x)$.
\end{definition}
Note that by Lemma \ref{01}, it is easy to verify that  belonging $[f]$ to $\pi_{H}^{qs}(X,x)$  is independent of the choice of representative $f$. 
\begin{theorem}\label{th3.2}
If  $H$  is a (normal) subgroup of $\pi_1 (X,x)$, then $\pi_{H}^{qs}(X,x)$ is a (normal) subgroup of $\pi_1 (X,x)$. In particular, $\pi_{1}^{qs}(X,x)$ is a normal subgroup of $\pi_1 (X,x)$.
\end{theorem}
\begin{proof}
First, assume that  $[f_1]$ and $[f_2]$ are arbitrary elements of $\pi_{H}^{qs}(X,x)$. Then $f_i \xrightarrow{close}h_i $ rel $\dot{I}$ for some $[h_i ]\in H$ for $i=1,2$.   Lemma \ref{02} implies that $f_1 *f_{2}^{-1}\xrightarrow{close}h_1 * h_{2}^{-1}$ rel $\dot{I}$. Since $H$ is a subgroup, $[h_1][h_2]^{-1}\in H$. Hence $[f_1][f_2]^{-1}\in \pi_{H}^{qs}(X,x)$. 
Now let $H$ be normal, and let $[f]\in \pi_H^{qs} (X,x)$ and $[g]\in \pi_1 (X,x)$ be arbitrary elements. By the definition of ?, $f\xrightarrow{close}h$ rel $\dot{I}$  for some $[h] \in H$. By applying Lemma \ref{02}, we have $g*f*g^{-1}\xrightarrow{close}g*h*g^{-1}$ rel $\dot{I}$, but $[g][h][g]^{-1} \in H$, since $H$ is normal. Hence $[g][f][g]^{-1}\in \pi_H^{qs} (X,x)$.
\end{proof}
By Definition \ref{def3.1n},  $H\subseteq \pi_H^{qs} (X,x)$, but the equality  does not hold in general. To see this, consider  $X$ as the wedge of  Harmonic archipelago and a line segment (Figure \ref{fig1}) for  which  $\{ e\} \neq \pi_1^{qs} (X,x)$ (see Remark \ref{rem3}). Moreover, the other elements of $\pi_H^{qs} (X,x)$ are generated by $H$ and $\pi_1^{qs}(X,x)$.

\begin{prop}\label{4444}
If $f\xrightarrow{close}g $ rel $\dot{I}$, then $[f*g^{-1}] \in \pi_1^{qs}(X,x)$. Moreover, for a subgroup $H$ of $\pi_1 (X,x)$
\[
\pi_{H}^{qs}(X,x )=H \pi_1^{qs}(X,x ).
\]
\end{prop}
\begin{proof}
Let $f\xrightarrow{close}g $ rel $\dot{I}$. By Lemma \ref{02}, $f*g^{-1}\xrightarrow{close}g*g^{-1} $. Therefore, $[f*g^{-1}] \in \pi_1^{qs}(X,x)$, since $g*g^{-1} \in [c_x]$. Clearly, $\pi_1^{qs}(X,x )\subseteq \pi_{H}^{qs}(X,x )$. So we have $\langle \pi_1^{qs}(X,x )\cup H\rangle \subseteq \pi_{H}^{qs}(X,x )$. On the other hand, let $[f]\in \pi_{H}^{qs}(X,x )$. Hence, $f\xrightarrow{close}h$ for some $[h]\in H$. Then $[f*h^{-1}]\in \pi_1^{qs}(X,x )$, but we have $[f]=[f*h^{-1}][h]\in \langle \pi_1^{qs}(X,x)\cup H\rangle$. Accordingly, $\pi_{H}^{qs}(X,x )\subseteq \langle \pi_1^{qs}(X,x )\cup H\rangle$. Since $\pi_{1}^{qs}(X,x)$ is a normal subgroup of $\pi_1 (X,x)$, $H \pi_1^{qs}(X,x )$ is a subgroup of $\pi_1 (X,x)$. Therefore, $\pi_{H}^{qs}(X,x ) = \langle \pi_1^{qs}(X,x )\cup H\rangle =H \pi_1^{qs}(X,x )$. 
\end{proof}

\begin{corol}\label{co3.4}
Let $H$ be a subset of $\pi_1 (X,x)$ and let $K$ be a subgroup of $\pi_1 (X,x)$ containing $\pi_{1}^{qs}(X,x)$ (or $\pi_1^{sp} (X, x)$).  Then  $H\subseteq K$ if and only if $\pi_{H}^{qs}(X,x)\subseteq K$. In particular, $\pi_K^{qs} (X, x) =K$.
\end{corol}

\begin{proof}
Trivially, if $\pi_{H}^{qs}(X,x)\subseteq K$, then $H\subseteq K$. To prove  the other direction, let $[f]\in \pi_{H}^{qs}(X,x)$. Then  $f\xrightarrow{close}h$ for some $[h]\in H$. By Proposition \ref{4444}, $[f][h]^{-1}\in \pi_{1}^{qs}(X,x)\subseteq K$ (by Proposition \ref{SP}, $[f][h]^{-1}\in \pi_{1}^{sp} (X,x)\subseteq K$). On other hand, by the hypothesis $H\subseteq K$, we have $[h]\in K$. Hence $[f]=[f][h]^{-1}[h]\in K$.
\end{proof}

For any subgroups $H$ and $K$ of $\pi_1 (X,x)$ with $\pi_{1}^{qs}(X,x) \subseteq H, K$, using Corollary \ref{co3.4}, we have
$H =K $ if and only if $ \pi_{H}^{qs}(X,x) = \pi_{K}^{qs}(X,x).$
 
\begin{corol}
Let $H \leq K\leq\pi_1 (X,x)$. If $\pi_H^{qs} (X, x) =H$, then $\pi_K^{qs} (X, x) =K$. 
\end{corol}

\begin{proof}
Assume that $H \leq K\leq\pi_1 (X,x)$ and that $\pi_H^{qs} (X, x) =H$. Then   Proposition \ref{4444} implies that $ H \pi_{1}^{qs}(X,x) = \pi_H^{qs} (X, x) =H$. Hence 
$\pi_{1}^{qs}(X,x) \subseteq H$. It follows from  $H \leq K$ that $\pi_{1}^{qs}(X,x) \subseteq K$. Therefore, $\pi_K^{qs} (X, x) =K$ by Corollary \ref{co3.4}.
\end{proof}

The group $\pi_1^s (X, x)$ depends on the choice of base point, but in the following proposition, we show that $\pi_1^{qs} (X, x)$, as a generalization of $\pi_1^s(X, x)$, does not depend on the base point in a given path component.

\begin{prop}
Let $H$ be a subset of $\pi_1 (X,x)$, and let $\lambda :I\longrightarrow X$ be  a path from $x$ to $y$. Then $\pi_{H}^{qs}(X,x)\cong \pi_{\lambda^{-1}H\lambda}^{qs}(X,y)$. 
In particular, $\pi_1^{qs}(X,x)\cong \pi_1^{qs}(X,y)$ for all $x$ and $y$  belonging to a path component of $X$. 
\end{prop}
\begin{proof}
Suppose $[f]\in \pi_{H}^{qs}(X,x )$. By Definition \ref{def3.1n}, $f\xrightarrow{close}h$ rel $\dot{I}$ for some $[h]\in H$.  By applying  Lemma \ref{02}, we have $\lambda^{-1}f\lambda \xrightarrow{close}\lambda^{-1}h\lambda$, where $ [\lambda^{-1}h\lambda]\in \lambda^{-1}H\lambda$. Therefore, $[\lambda^{-1}f\lambda]\in \pi_{\lambda^{-1}H\lambda}^{qs}(X,y)$. Then  we can define $\beta_{\lambda}:\pi_{H}^{qs}(X,x )\longrightarrow \pi_{\lambda^{-1}H\lambda}^{qs}(X,y)$ by $\beta_{\lambda}([f])=[\lambda^{-1}f\lambda ]$,  which is a group isomorphism. 
\end{proof}

The following proposition implies that $\pi_1^{qs}$ can be considered as a functor from $hTop_*$, the category of pointed topological spaces, to $Groups$, the category of groups.

\begin{prop}
Let $\phi :(X,x )\to (Y,y)$ be a continuous map and let $H\subseteq \pi_1 (X,x)$. Then $\phi_* :\pi_{H}^{qs}(X,x)\to \pi_{\phi_* (H)}^{qs}(Y,y)$ defined by $\phi_* ([f])=[\phi f]$ is a homomorphism.
\end{prop}
\begin{proof}
Let $[f]\in \pi_{H}^{qs}(X,x)$. Then $f\xrightarrow{close}h\; rel\; \dot{I}$ for some $[h]\in H$.  By Proposition \ref{pr2.3}, $\phi f\xrightarrow{close}\phi h\; rel\; \dot{I}$. This means that  $[\phi f]\in \pi_{\phi_* (H)}^{qs}(Y,y)$. One can easily check that $\phi_*$ is a homomorphism.
\end{proof}

\begin{remark}
As mentioned in Section 1,  $\pi_1^s(X, x)$ may admit different structures at different points. In order to have a subgroup not depending on the base point, Virk \cite{3} introduced   $\pi^{sg}_1 (X, x)$, as the subgroup generated by the set
$$
\{ [\beta *\alpha *\beta^{-1}]\; |\; \beta \in P(X,x), \; [\alpha ]\in \pi_1^s (X,\beta (1)) \} , 
$$
where $P(X, x)$ is the space of all paths in $X$ with initial point $x$. In general, $\pi_1^s (X,x)\subseteq \pi_{1}^{sg}(X,x)$, however, there exist spaces, namely the Harmonic archipelago $HA$ for which $\pi_1^s (HA,x)$ is trivial but $\pi_{1}^{sg}(HA, x)$ is not, whenever $x$ is any point except the origin. In the same manner,  we can study $\pi_{H}^{qsg}(X,x )$ as the subgroup generated by the following set:
$$
\{ [\beta *\alpha *\beta^{-1}]\; |\; \beta \in P(X,x), \; [\alpha ]\in \pi_{\beta^{-1}H\beta}^{qs}(X,\beta (1) ) \} .
$$
We show that  $\pi_{H}^{qs}(X,x )=\pi_{H}^{qsg}(X,x )$. Clearly $\pi_{H}^{qs}(X,x )\subseteq \pi_{H}^{qsg}(X,x )$.  Assume that $[\beta *\alpha *\beta^{-1}]$ is an arbitrary generator of $\pi_{H}^{qsg}(X,x )$, where $\beta \in P(X,x )$ and $[\alpha ]\in \pi_{\beta^{-1}H\beta }^{qs}(X,\beta (1) )$. Since $[\alpha ]\in \pi_{\beta^{-1}H\beta}^{qs}(X,x )$, it follows that $\alpha \xrightarrow{close}\beta^{-1}h\beta$ for some $[h]\in H$. Then $\beta *\alpha *\beta^{-1}\xrightarrow{close}\beta *\beta^{-1} *h*\beta *\beta^{-1}$, where $[\beta *\beta^{-1} *h*\beta *\beta^{-1}]=[h]\in H$. So $[\beta *\alpha *\beta^{-1}]\in \pi_{H}^{qs}(X,x )$. Thus $\pi_{H}^{qsg}(X,x )\subseteq \pi_{H}^{qs}(X,x )$.  
\end{remark}

\begin{remark}\label{rem3}
The inclusion  $\pi_1^{s}(X,x)\leqslant \pi_1^{qs}(X,x)$ holds trivially, however, there exist spaces $X$ for which $\pi_1^s (X,x)=1$ while $\pi_1^{qs} (X,x)\neq 1$. Take $X$ as the wedge of  Harmonic archipelago and a line segment (Figure \ref{fig1}). At point $x$, there is no non-trivial small loop,  and hence $\pi_1^s(X, x) = 1$, but  $f*g*f^{-1} \xrightarrow{close} f*f^{-1} \simeq c_x$. Thus  $[f*g*f^{-1}]\in \pi_1^{qs} (X,x)$ and moreover $[f*g*f^{-1}] \neq 1$. Therefore $\pi_1^s(X, x) = 1$ and $\pi_1^{qs}(X, x) \neq 1$.
\end{remark}
\begin{figure}[tb]
\centering
\includegraphics[scale=0.4]{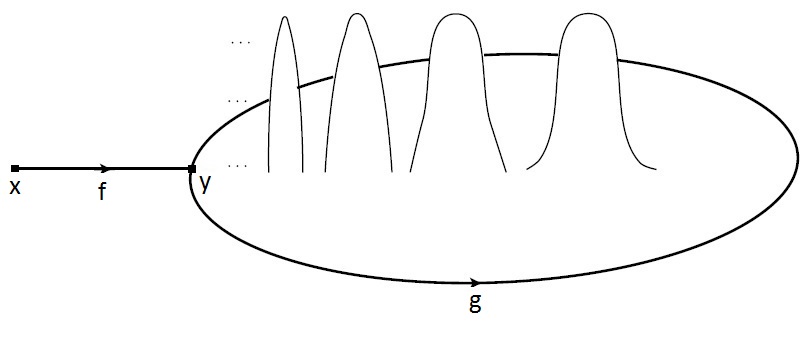}
\caption{The wedge of  Harmonic archipelago and line segment}\label{fig1}
\end{figure}

Proposition \ref{SP} presents a relationship between homotopical closeness and the Spanier subgroup, which helps us to find the location
of  subgroup $\pi_1^{qs} (x,x)$ in the chain $\pi_{1}^{s}(X,x)\leqslant \pi_{1}^{sg}(X,x)\leqslant \pi_{1}^{sp}(X,x)$ as follows. 
\begin{prop}\label{relation}
For a topological  space $X$ and any $x\in X$, 
\begin{equation}\label{eqmain}
 \pi_{1}^{sg}(X,x )\leqslant \pi_1^{qs}(X,x )\leqslant \pi_{1}^{sp}(X,x ).
\end{equation}
\end{prop}
\begin{proof}
First, assume that $[\alpha *\beta *\alpha^{-1}]$ is an arbitrary generator of $\pi_{1}^{sg}(X,x )$, where $\alpha \in P(X,x )$ and $[\beta ]\in \pi_{1}^{s}(X,\alpha (1))$. Then  $\beta \xrightarrow{close}c_{\alpha (1)}\; rel\; \dot{I}$, which implies that  $\alpha *\beta *\alpha^{-1}\xrightarrow{close}\alpha *c_{\alpha (1)}*\alpha^{-1}\; rel\; \dot{I}$. Since  $\alpha *c_{\alpha (1)}*\alpha^{-1}\in [c_{x}]$, so  $[\alpha *\beta *\alpha^{-1}]\in \pi_1^{qs}(X,x )$. 

Now, assume that $[f]\in \pi_1^{qs}(X,x )$. Then $f\xrightarrow{close}g$ for some $g\in [c_{x}]$. By Proposition \ref{SP}, $[f]=[f*g^{-1}]\in \pi_{1}^{sp}(X,x )$. Thus $\pi_1^{qs}(X,x )\leqslant \pi_{1}^{sp}(X,x )$.
\end{proof}

\begin{remark}\label{rem1}
\begin{enumerate}
\item
Consider the space $Z$ in Remark \ref{000} (Figure \ref{fig2}).  Clearly $f\xrightarrow{close}f' $ rel $\dot{I}$ and $f'\in [c_x]$.
 Then $1\neq [f]\in \pi_1^{qs} (Z,x)$ while $\pi_1^{sg}(Z,x)=1$, because $Z$ is semilocally simply connected at any point. This shows that the inclusion $ \pi_{1}^{sg}(X,x )\leqslant \pi_1^{qs}(X,x )$ can be strict. 
\item
Consider  the space $Y'$ defined by Fischer et al. \cite{On}.  By \cite[Proposition 3.2 and Theorem 3.7]{On},  $Y'$ is homotopically path Hausdorff and then $\pi_{1}^{sp}(Y', x)\neq 1$, but  by Theorem \ref{s}, $\pi_1^{qs} (Y',x)=1$. Putting $X = Y'$ implies that the inclusion $\pi_1^{qs}(X,x )\leqslant \pi_{1}^{sp}(X,x )$ can be strict.  
\end{enumerate}
\end{remark}

It is well known that a space $X$ is homotopically Hausdorff if and only if $\pi_1^s(X, x) =1$ for all $x \in X$. Due to  the fact that the notion of closeness is a generalization of the smallness, Virk generalized the above statement in \cite[Proposition 48]{2} as follows: A locally path connected metric space $X$ is homotopically path Hausdorff if and only if there are no (non-homotopic) paths $f,g:I \to X$ with $f \xrightarrow{close} g$. For a topological space $X$, if there are non-homotopic paths $f,g:I \to X$ with $f \xrightarrow{close} g$, then by Proposition \ref{4444}, $[f*g^{-1}] \in \pi_{1}^{qs}(X,x)$, which implies that $ \pi_{1}^{qs}(X,x) \neq 1$. If $ \pi_{1}^{qs}(X,x) \neq 1$, then there is a non-null-homotopic loop $f$ in $X$ such that $[f] \in \pi_{1}^{qs}(X,x)$. So  $f\xrightarrow{close}g$ for some $g\in [c_{x}]$. Hence, for a topological space $X$, there are no non-homotopic paths $f,g:I \to X$ with $f \xrightarrow{close} g$ if and only if $ \pi_{1}^{qs}(X,x) = 1$. Since, for non-homotopic paths, Definition \ref{ASL} and Definition 58 in \cite{2} coincide on locally path connected spaces, it follows from \cite[Proposition 48]{2} that a locally path connected metric space $X$ is homotopically path Hausdorff if and only if $ \pi_{1}^{qs}(X,x) = 1$. In the following theorem, we prove a more general version of the above fact, using the quasi-small loop group.

\begin{theorem}\label{thm100}
Let $X$ be a topological space  and let $x\in X$. 

\begin{enumerate}
\item
If $\pi_1^{qs} (X,x)=1$, then $X$ is homotopically path Hausdorff;
\item
If $X$ is homotopically path Hausdorff with respect to a subgroup $H  \leq \pi_1(X,x)$, then $\pi_{H}^{qs} (X,x)=H$. Moreover, if $X$ is homotopically path Hausdorff, then $\pi_{H}^{qs} (X,x)=H$ for every $H \leq \pi_1(X,x)$.
\end{enumerate}

\end{theorem}
\begin{proof}
$1$. Suppose that $X$ is not homotopically path Hausdorff. Then by Definition \ref{def10}, there exist paths $\alpha, \beta \in P(X, x)$ with  $\alpha (1)=\beta (1)$ and $[\alpha] \neq [\beta]$ such that  for every  partition $0 = t_0 < t_1 < t_2 < \cdots < t_n = 1$ and every sequence of open sets $U_1,U_2, \ldots ,U_n$ with $\alpha ([t_{i-1}, t_i ]) \subset U_i$, there exists a path  $\gamma : [0, 1] \to 
X$  satisfying $\gamma ([t_{i-1}, t_i ]) \subset U_i$ for $1 \leq  i \leq  n$ and $\gamma (t_i ) = \alpha (t_i )$ for
$0 \leq i \leq n$ such that  $[\gamma ] = [ \beta]$. By Definition \ref{ASL}, $\beta \xrightarrow{close}\alpha$  rel  $\dot{I}$. Now by Lemma \ref{02}, $\alpha *\beta^{-1}\xrightarrow{close} \alpha*\alpha^{-1}$  rel  $\dot{I}$. This means that $[\alpha *\beta^{-1}]\in \pi_{1}^{qs}(X,x)$. Therefore $\pi_{1}^{qs}(X,x) \neq 1$ since $[\alpha] \neq [\beta]$.

$2$. Suppose that  $X$ is  homotopically path Hausdorff relative to $H$. By contrary, assume that there exists $[f]\in \pi_H^{qs} (X,x)\setminus H$. Then $f\xrightarrow{close}h$ rel $\dot{I}$ for some $[h]\in H$. Since $[f]\not \in H$,  $[h*f^{-1}]\not \in H$. Now by the hypothesis,   there are a partition $0 = t_0 < t_1 < t_2 < \cdots < t_n = 1$ and a sequence of open sets $U_1,U_2, \ldots ,U_n$ with $h ([t_{i-1}, t_i ]) \subset U_i$, such that for any path  $\gamma : [0, 1] \to  X$  satisfying $\gamma ([t_{i-1}, t_i ]) \subset U_i$ for $1 \leq  i \leq  n$ and $\gamma (t_i ) = h (t_i )$ for $0 \leq i \leq n$, we have $[\gamma * f^{-1}] \not \in H$. This implies that $\gamma \not \simeq f$  rel  $\dot{I}$,  which contradicts with the fact that $f\xrightarrow{close}h$ rel $\dot{I}$. Thus $\pi_{H}^{qs} (X,x)=H$. Now let $X$ be homotopically path Hausdorff. Then $\pi_{1}^{qs} (X,x)=1$, which implies that $\pi_{H}^{qs} (X,x)=H$ for every $H \leq \pi_1(X,x)$ by Corollary \ref{co3.4}.
\end{proof}

The next theorem presents several equivalent conditions for the property of  homotopically path Hausdorffness. 
Theorem \ref{s} part (1)$\implies$(2) is obtained by \cite[Proposition 48]{2}  for locally path connected metric  spaces. Moreover, Virk presented the example $C(\mathbb{S}^1, \{0\})$ to show that the property of  locally path connectedness cannot be omitted \cite[p. 370]{2}. In the following, by adding the condition $\gamma(t_i) = g(t_i)$ in Definition \ref{ASL}, we see that Proposition 48 in \cite{2} holds for any arbitrary spaces.

\begin{theorem}\label{s}
Let $X$ be a topological space  and let $x\in X$. Then the following statements are equivalent:
\begin{enumerate}
\item
$X$ is homotopically path Hausdorff;
\item
$\pi_1^{qs} (X,x)=1$;
\item
 for all $[f]\in \pi_1 (X,x)$, if $f\xrightarrow{close}g$ rel $\dot{I}$, then $[f]=[g]$;
\item
 for all $[f]\in \pi_1 (X,x)$, if $f\xrightarrow{close} g_1$ rel $\dot{I}$ and $f\xrightarrow{close} g_2$ rel $\dot{I}$, then $[g_1 ]=[g_2 ]$.
\end{enumerate}
\end{theorem}
\begin{proof}
$(1) \Leftrightarrow (2)$. It follows from Theorem \ref{thm100} for the trivial subgroup $H=1$ of $\pi_1 (X,x)$.

$(2) \Leftrightarrow (3)$. Assume that $\pi_1^{qs}(X,x)=1$ and $f\xrightarrow{close} g$ rel $\dot{I}$ for $[f]\in \pi_1 (X,x)$. Then by Lemma \ref{02}, $f*g^{-1}\xrightarrow{close}g*g^{-1}$ rel $\dot{I}$, where $g*g^{-1}\in [c_{x}]$. Hence $[f*g^{-1}]\in \pi_1^{qs}(X,x)=1$ and so $[f]=[g]$. Thus the condition $(3)$ holds.

Conversely, suppose that for all $[f]\in \pi_1 (X,x)$, if $f\xrightarrow{close}g$ rel $\dot{I}$, then $[f]=[g]$.  Assume that  $f\in \pi_1^{qs}(X,x)$. Then $f\xrightarrow{close}f'$ rel $\dot{I}$, for some $f' \in [c_{x}]$. Then by the hypothesis, we have $[f]=[f']=[c_{x}]$ and so, $\pi_1^{qs}(X,x)=1$.

$(3) \Leftrightarrow (4)$. Assume that for all $[f]\in \pi_1 (X,x)$, if $f\xrightarrow{close}g$ rel $\dot{I}$, then $[f]=[g]$. Also, assume that  $f\xrightarrow{close} g_1$ rel $\dot{I}$ and $f\xrightarrow{close} g_2$ rel $\dot{I}$ for $[f]\in \pi_1 (X,x)$. Then by the hypothesis,  $[g_1 ]=[f]=[g_2 ]$.

Conversely, assume that for all $[f]\in \pi_1 (X,x)$, the conditions   $f\xrightarrow{close} g_1$ rel $\dot{I}$ and $f\xrightarrow{close} g_2$ rel $\dot{I}$ imply  $[g_1 ]=[g_2 ]$.  Suppose that $f\xrightarrow{close} g$ rel $\dot{I}$ for some $[f]\in \pi_1 (X,x)$.  Since $f\xrightarrow{close} f$ rel $\dot{I}$, by the hypothesis  $[f]=[g]$. 
\end{proof}

By Theorem \ref{s}, we show the difference between Definition \ref{ASL} and Definition 58 in \cite{2} on non-locally path connected spaces.

\begin{example}\label{ex2.2}
Consider the space $C(\mathbb{S}^1 ,\{ 0\})$  (see Figure \ref{fig4}), constructed in \cite[p. 370]{2}, which  is recalled in Remark \ref{000}. The space $C(\mathbb{S}^1 ,\{ 0\})$ is  not locally path connected \cite[p. 370, Remark]{2}. Moreover, $C(\mathbb{S}^1, \{0\})$ is homotopically path Hausdorff  with loops that are close together \cite[p. 370]{2}, namely $f$ and $g$, in the sense of \cite[Definition 58]{2}. If $f$ and $g$ are homotopically close in the sense of Definition \ref{ASL}, then by Theorem \ref{s} (3), $[f] = [g]$ or equivalently $f \simeq g$ rel $\dot{I}$, which contradicts with \cite[Definition 58]{2}. Therefore $C(\mathbb{S}^1, \{0\})$ has no homotopically close loops in the sense of Definition \ref{ASL}, except homotopic ones.

\begin{figure}
\centering
\includegraphics[scale=.35]{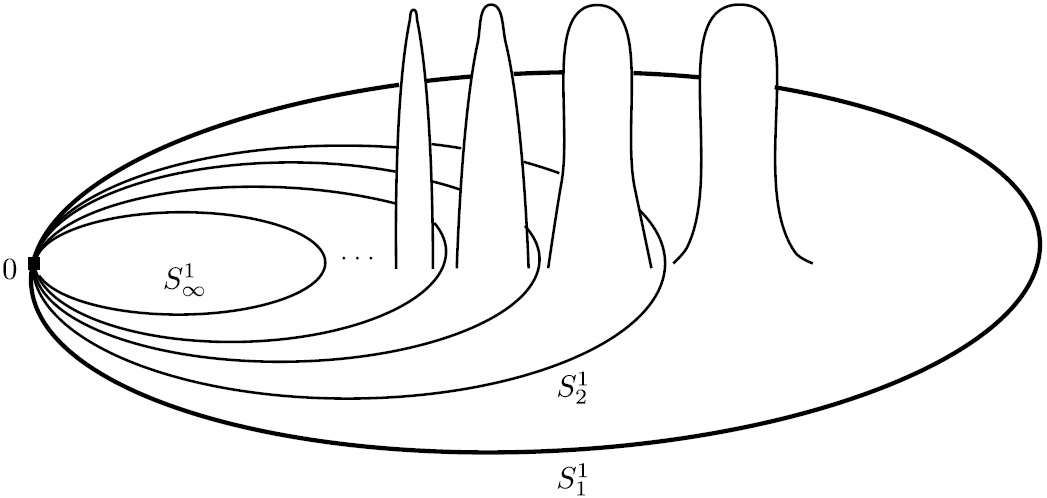}
\caption{The space $C(\mathbb{S}^1, \{0\})$}\label{fig4}
\end{figure} 
\end{example}


Theorem \ref{s} (1) $\Rightarrow$ (2) proves that the converse statement of part (1) of Theorem \ref{thm100} holds for every space $X$. Now there is a natural question whether the converse statement of part (2) of Theorem \ref{thm100} holds.

\begin{problem}\label{pr3.15n}
Under which conditions for a space $X$ and a subset $H$  of $\pi_1(X, x)$, if $\pi_H^{qs} (X, x) = H$, then we can conclude that  $X$ is homotopically path Hausdorff relative to $H$? 
\end{problem}

\section{Quasi-small loop group and quasi-topological fundamental group}

There are some relationships, indicated in Remark \ref{300}, between the quasi-small loop subgroup $\pi_H^{qs}(X, x)$ and the quasi-topological fundamental group $\pi_1^{qtop} (X, x)$ studied in Brazas-Fabel \cite{1} as the quotient of $P(X,x)$, the set of paths in $X$ starting at $x$ equipped with the compact open topology. Also, Brazas-Fabel \cite{1} proved that  for a locally path connected space $X$, the quasi-topological fundamental group $\pi_1^{qtop} (X, x)$ satisfies the $T_1$ separation axiom if and only if $X$ is homotopically path Hausdorff. Now, we find an equivalent condition for homotopical path Hausdorffness by the quasi-small loop group.

\begin{remark}\label{300}
Let $X$ be a locally path connected space and  let $H\subsetneqq \pi_1 (X, x)$. 
\begin{enumerate}
\item
If $H$ is closed in $\pi_{1}^{qtop}(X,x)$, then $\pi_H^{qs} (X,x)=H$.
By \cite[Lemma 9]{1}, $H$ is closed in $\pi_{1}^{qtop}(X,x)$ if and only if  $X$ is homotopically path Hausdorff relative to $H$. By Theorem \ref{thm100}, if $X$ is homotopically path Hausdorff relative $H$, then $\pi_{H}^{qs} (X,x)=H$. 

\item 
$\pi_{1}^{qtop}(X, x)$ is $T_1$ if and only if $\pi_1^{qs} (X, x)=1$.
By item (1) and Theorem \ref{thm100}, $\pi_1^{qs}(X, x) = 1$ if and only if the identity is closed in $\pi_1^{qtop}(X, x)$, which is equivalent to be $T_1$ for quasi-topological groups. 

\end{enumerate}

\end{remark}

\begin{figure}
\centering
\includegraphics[scale=0.6]{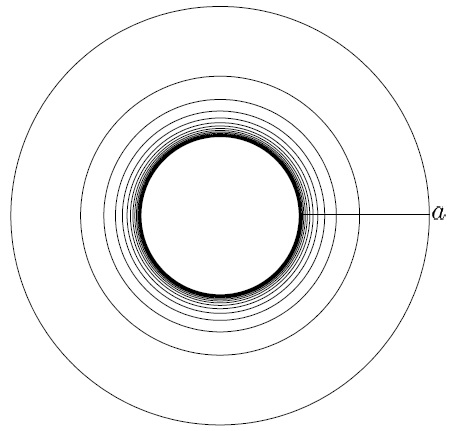}
\caption{The space $Y$}\label{fig3}
\end{figure}

\begin{remark}
The local path connectedness is an essential condition  in Remark \ref{300}. As an example for item (2), consider the space $X$ introduced in Example 2.9 of Torabi-Pakdaman-Mashayekhy \cite{31} as follows: Let $Y_1 = \{ (x, y) \in \mathbb{R}^2 \; |\; x^2 + y^2 = (1/2 + 1/n)^2,\; n \in \mathbb{N}\}$, let 
$Y_2 = \{ (x, y) \in \mathbb{R}^2 \; |\; x^2+y^2 = 1/4\}\cup \{ (x, 0) \in \mathbb{R}^2 \; |\;  1/2 \leq x \leq 3/2\}$, and let $Y = Y_1\cup Y_2$ with $x = (3/2, 0) = a$ as the base point (see Figure \ref{fig3}). Let $f_i : \mathbb{S}^1 \to S_i$ be a
homeomorphism from the 1-sphere into $Y$ such that $f_i ((1, 0)) = (1/2 + 1/i, 0)$,
where $S_i  = \{ (x, y) \in \mathbb{R}^2 \; |\;  x^2 + y^2 = (1/2 + 1/i)/2\}$, for every $i \in \mathbb{N}$. Put $X_0 = Y$,
and let $X_i = X_{i-1} \cup_{f_i} C_i$, where $C_i$ is a cone over $\mathbb{S}^1$ with height 1, be the
space obtained by attaching the cone $C_i$ to $X_{i-1}$ via $f_i$ for all $i \in \mathbb{N}$. Consider $X=\bigcup_{i\in \mathbb{N}}X_i$. Clearly, $X$ is path connected and semilocally simply connected, but it is not locally path connected. Then $\pi_1^{sp}(X,x)=1$. Hence $\pi_1^{qs}(X,x)=1$ by Proposition \ref{relation}. On the other hand, $\pi_1^{qtop}(X,x)$ is an indiscrete topological group by Remark 2.11 of Torabi-Pakdaman-Mashayekhy \cite{31}. Therefore $\pi_{1}^{qtop}(X,x)$ is not $T_1$.
\end{remark}

Homotopical Hausdorffness holds if and only if the path space is Hausdorff; see \cite{BroDyd}. For the relative cases, homotopicall Hausdorffness is equivalent to a certain quotient space be Hausdorff \cite{BabMas3, FisZas}.
Now because of the name of the property of homotopically path Hausdorff and other similar results, it is interesting to know if there is any topology on the fundamental group or on the path space, which is suitable as follows.

\begin{problem}
Is there any topology on the fundamental group or path space that satisfies the following equivalence? 
\begin{itemize}
\item
a space $X$ is homotopically path Hausdorff relative to $H$ if and only if either the quotient space $\frac{\pi_1(X, x)}{H}$ or a quotient of the path space is Hausdorff. 
\end{itemize}
\end{problem}

Also by Remark \ref{300}, we can restate Problem \ref{pr3.15n} as follows.

\textit{
Under which conditions for a space $X$ and subset $H$ of $\pi_1(X, x)$, if  $\pi_H^{qs} (X, x) = H$, then $H$ is closed in $\pi^{qtop}_1(X, x)$?}


\begin{thebibliography}{99}

\bibitem{BabMas3}
 A. Babaee, B. Mashayekhy, H. Mirebrahimi, H. Torabi, M. Abdullahi Rashid, S.Z. Pashaei, \textit{On topological homotopy groups and its relation to Hawaiian groups},  {Hacettepe J. Math. Stat.} {\bf 49} (4), 1437--1449, 2020.


%
%
%

\bibitem{1} J. Brazas, P. Fabel, \textit{On fundamental group with the quotient topology}, Homotopy Rel. Struc. \textbf{10},  71--91, 2015. 

\bibitem{BroDyd}
N. Brodskiy, J. Dydak, B. Labuz, A. Mitra, \textit{Covering maps for locally path connected spaces}, Fund. Math. \textbf{218}, 13--46, 2012.

\bibitem{Con} G. Conner, M. Meilstrup, D. Repov\v{s}, A. Zastrow, M. \v{Z}eljko,  \textit{On small homotopies of loops}, Topology Appl. \textbf{155}, 1089--1097, 2008. 

\bibitem{Can} J.W. Cannon, G.R. Conner, \textit{On the fundamental groups of one-dimensional spaces}, Topology Appl. \textbf{153}, 2648--2672, 2006. 



\bibitem{EdaKar1n}
K. Eda, U.H. Karimov, D. Repov\v{s}, \textit{On (co)homology locally connected spaces}, Topology Appl. \textit{120}:3, 397--401, 2002.


\bibitem{EdaKar2n}
K. Eda, U.H. Karimov and D. Repov\v{s}, \textit{A construction of noncontractible simply connected cell-like two-dimensional Peano continua}, Fund. Math. \textbf{195}: 3, 193--203, 2007.

\bibitem{EdaKar3n}
K. Eda, U.H. Karimov and D. Repov\v{s}, \textit{A nonaspherical cell-like 2-dimensional simply connected continuum and related constructions}, Topology Appl. \textbf{156}:3, 515--521, 2009.


\bibitem{EdaKar4n}
K. Eda, U.H. Karimov and D. Repov\v{s}, \textit{On the second homotopy group of SC(Z)}, Glas. Mat. \textbf{44} (64):2, 493--498, 2009.


\bibitem{EdaKar5n}
K. Eda, U.H. Karimov and D. Repov\v{s}, \textit{On the singular homology of one class of simply-connected cell-like spaces}, Mediterr. J. Math. \textbf{8}:2, 153--160, 2011.

\bibitem{EdaKar}
K. Eda, U.H. Karimov, D. Repov\v{s}, \textit{On $2$-dimensional nonasapherical cell-like Peano continua: A simplified approach}, Mediterr. J. Math. \textbf{10}:1, 519--528, 2013.

\bibitem{EdaSnake}
K. Eda, U.H. Karimov, D. Repov\v{s}, A. Zastrow, \textit{On Snake cones, alternating cones and related constructions}, Glas. Mat. \textbf{48}:1, 115--135, 2013.







\bibitem{Far}
M. Farber, \textit{Topology of robot motion planning. Morse theoretic methods in nonlinear analysis and in symplectic topology, }185--230, NATO Sci. Ser. II Math. Phys. Chem., 217, Springer, Dordrecht, 2006. 


\bibitem{On} 
H. Fischer, D. Repov\v{s}, \v{Z}. Virk, A. Zastrow,  \textit{On semilocally simply connected spaces}, Topology Appl.  \textbf{158},  397--408, 2011. 




\bibitem{FisZas}
H. Fischer, A. Zastrow, \textit{Generalized universal coverings and the shape group}, Fund. Math.  {\bf 197}, 167--196, 2007.



\bibitem{KarRep}
U. Karimov, D. Repov\v{s}, W. Rosicki, A. Zastrow, \textit{On two-dimensional planar compacta not homotopically equivalent to any one-dimensional compactum}, Topology Appl. \textbf{153}:2-3, 284--293, 2005.


\bibitem{KarRep}
U.H. Karimov and D. Repov\v{s}, \textit{Hawaiian groups of topological spaces}, Russian Math. Surveys \textbf{61}:5, 987--989, 2006.


\bibitem{KarRephom}
U. H. Karimov and  D. Repov\v{s}, \textit{On the homology of the Harmonic archipelago}, Cent. Eur. J. Math. {\bf 
10}, 863--872, 2012.

\bibitem{MalRep}
J. Male\v{s}i\v{c}, D. Repov\v{s}, W. Rosicki, A. Zastrow, \textit{On continua with homotopically fixed boundary}, Topology Appl. \textbf{154}:3, 639--654, 2007.


\bibitem{BabMas4}
 B. Mashayekhy, H. Mirebrahimi, H. Torabi and A. Babaee, {\it On small $n$-Hawaiian loops},  Mediterr. J. Math.  \textbf{17}, 202, 2020.

\bibitem{PasMas}
S.Z. Pashaei, B. Mashayekhy, H. Torabi, M. Abdullahi Rashid, \textit{Small Loop Transfer Spaces with respect to Subgroups of Fundamental Groups}, Topology Appl. {\bf 232}, 242--255, 2017.



\bibitem{31} H. Torabi, A. Pakdaman and B. Mashayekhy, \textit{Topological fundamental groups and small generated coverings}, Math. Slovaca \textbf{65}, No. \textbf{5}, 1153--1164, 2015. 

\bibitem{2} \v{Z}. Virk, \textit{Homotopical smallness and closeness}, Topol. Appl. \textbf{158}, 360--378, 2011. 
\bibitem{3} \v{Z}. Virk,  \textit{Small loop spaces}, Topology Appl. \textbf{157},  451--455, 2010. 




\end{thebibliography}
\end{document}